\documentclass[draft,11pt]{article}

\usepackage{amsfonts,amsmath,amsthm,amscd,amssymb,latexsym,cite,verbatim,texdraw,floatflt,caption2,pb-diagram}

   \usepackage[T2A]{fontenc} 
    \usepackage[cp1251]{inputenc}
    \usepackage[ukrainian]{babel}

\usepackage[mathscr]{eucal}

\newtheorem{theorem}{Теорема}[section]
\newtheorem{lemma}{Лема}[section]

\theoremstyle{definition}

\renewcommand{\abstract}{\textbf{Анотація. }\medskip}

\numberwithin{equation}{section}

\begin{document}
\Large

\noindent УДК 517.5

\bigskip\noindent
{\bf В.В. Савчук$^{1}$, С.О. Чайченко$^{2}$, А.Л. Шидліч$^{1, 3}$}

$^{1\ }$Інститут математики Національної академії наук України, м. Київ

$^{2\ }$Державний вищий навчальний заклад “Донбаський державний педагогічний університет”, м. Слов'янськ

$^{3\ }$Національний університет біоресурсів і природокористування України, м. Київ

\bigskip\noindent
{\bf  ЕКСТРЕМАЛЬНІ ЗАДАЧІ ВАГОВОГО НАБЛИЖЕННЯ \\ НА ДІЙСНІЙ ОСІ }

\bigskip
{\noindent }
\bigskip

{ \bf Ключові слова:} раціональна функція; система Такенаки-Мальмквіста, найкраще наближення.
\bigskip

{ \bf AMS:} 46E30, 42A10,41A17,41A20,41A25,41A27, 41A30.
\bigskip

\textbf{Анотація.} 
 Обчислено точне значення та знайдено поліном найкращого вагового поліноміального наближення 
 ядер вигляду $\frac {A+Bx}{(x^2+\lambda^2)^2}$, де $A,B\in {\mathbb R}$, $\lambda>0$, 
 в середньо-квадратичній метриці.

\section{Вступ.} В роботі М.\,М.~Джрбашяна \cite{Djrbashan_1974} було розвинено метод, який дозволяє отримувати розв'язки екстремальних задач про найкраще вагове поліноміальне наближення ядра Коші $\frac 1{\omega-x}$, $\mathop{\rm Im} \omega\not=0$,
на всій дійсній осі ${\mathbb R}$ у рівномірній та середньо квадратичній метриці, у випадку, коли вага є поліномом із заданими полюсами. Цей метод, зокрема, використовував зображення даного ядра за допомогою відрізків його ряду Фур'є за ортогональною на всій осі системою раціональних функцій $\{\Phi_n(z)\}_{n=1}^\infty$ із заданою множиною полюсів $\{a_k\}_{k=1}^{\infty}$, $(\mathop{\rm Im} a_k>0)$.

В роботі А.\,А.~Восканяна  \cite{Voskanian}, використовуючи даний метод та деякі інші результати з
\cite{Djrbashan_1974},     було розв'язано аналогічні задачі про найкраще вагове поліноміальне наближення ядер вигляду $\frac {A+Bx}{x^2+\lambda^2}$, $A,B\in {\mathbb R}$, $\lambda>0$, які були вперше розглянуті С.\,Н.~Бернштейном.

В цій роботі знаходяться розв'язки аналогічних екстремальних задач про найкраще вагове поліноміальне наближення
ядер вигляду $\frac {A+Bx}{(x^2+\lambda^2)^2}$, $A,B\in {\mathbb R}$, $\lambda>0$, в середньо-квадратичній метриці.

 {\bf 1. Ортогональна система раціональних функцій на дійсній осі.}
 Нехай  ${\bf a}:=\{a_k\}_{k=1}^{\infty}$, $(\mathop{\rm Im} a_k>0)$ --- довільна послідовність комплексних чисел,
  які лежать в верхній півплощині $\mathbb C_+:=\{z\in\mathbb C : \mathop{\rm Im}z>0\}$. Розглянемо систему функцій
\begin{equation}\label{Phi-system} 
    \Phi_1(z):=\frac{\sqrt{\mathop{\rm Im} a_1}}{z-\overline a_1}, \quad
    \Phi_n(z):=\frac{\sqrt{\mathop{\rm Im} a_n}}{z-\overline a_n}B_n(z),~n=2,3,\ldots,
\end{equation}
де
$$
    B_1(z):=1, \quad B_n(z):=\prod_{k=1}^{n-1}\chi_k
    \frac{z-a_k}{z-\overline a_k},~
    \chi_k:=\frac{|1+a^2_k|}{1+a^2_k},~n=2,3,\ldots.
$$
--- $n$-добуток Бляшке з нулями в точках $a_k,$ $k=1,...n$.

Систему функцій $\{\Phi_n(z)\}_{n=1}^\infty$ запровадив М.М.~Джрбашян \cite{Djrbashan_1974} за аналогією, як це зробили С.~Такенака і Ф.~Мальмквіст (див.  \cite{Takenaka, Malmquist}) у випадку простору Гарді $H_2$ в одиничному крузі. Зокрема, в \cite{Djrbashan_1974} показано, що система $\{\Phi_n(z)\}_{n=1}^\infty$ є ортонормованою на дійсній осі $\mathbb R$, тобто
$$
    \frac{1}{\pi} \int\limits_{-\infty}^\infty
    \Phi_n (x)\overline{\Phi_m(x)}~\mathrm{d}x=
    \left\{\begin{matrix}0, \hfill& n\not=m,\\
                    1,\hfill& n=m,
                    \end{matrix}\right.\quad n,m=1,2,\ldots.
$$
і справджується лема.

\begin{lemma}\label{L-Dzhrbsh-ident} Для довільних $z, \zeta\in\mathbb C,$ $z\not=\overline\zeta$ і $n,m= 1,2,\ldots,$
 справджується рівність
\begin{equation}\label{Dzhrbashyan-identity} 
    \frac{1}{2{\mathrm i}(\overline\zeta-z)} =\sum_{k=1}^{n}\overline{\Phi_k(\zeta)}\Phi_k(z)+
    \frac{\overline{B_n(\zeta)}B_n(z)}{2{\mathrm i}(\overline\zeta-z)}.
\end{equation}
\end{lemma}

Рівність (\ref{Dzhrbashyan-identity}) є аналогом формули Христофеля-Дарбу для ортогональних многочленів і відома в літературі під назвою тотожності М.М.~Джрбашяна. Доведення леми \ref{L-Dzhrbsh-ident} також можна знайти в \cite{Savchuk_CH_umg_2014}.

Аналітична у півплощині $\mathbb{C}_+$ функція $f(z)$ належить класу Гарді $H_2(\mathbb{C}_+),$ якщо
\[
    \|f\|_{H_2(\mathbb{C}_+)}:= \sup_{y>0}\left(\frac{1}{\pi} \int_{-\infty}^\infty|f(x+iy)|^2~\mathrm{d}x\right)^{1/2}<\infty.
\]

Відомо \cite[теорема 2]{Krilov-1939} (див. також \cite[с. 291]{Mashreghi-Reperes-Theorem}), що кожна функція $f$ з простору $H_2(\mathbb{C}_+)$ має майже скрізь на дійсній осі $\mathbb R=\partial\mathbb C_+$ недотичні граничні значення
$$
    f(x)=\lim_{y \to 0} f(x+{\mathrm i}y),
$$
(за якими залишимо те ж саме позначення $f$), які належать простору $L_2(\mathbb R),$ причому
\[
    \|f\|_{H_2(\mathbb{C}_+)}=
    \|f\|_2:=\left(\frac{1}{\pi}\int_{-\infty}^\infty|f(t)|^2~\mathrm{d}t\right)^{1/2}
\]
і справджується інтегральна формула Коші
\begin{equation}\label{Coushy-form}
    f(z)=\frac{1}{2\pi {\mathrm i}} \int_{-\infty}^\infty \frac{f(t)}{t-z}~\mathrm{d}t,
    \quad \forall z\in\mathbb C_+.
\end{equation}

Наступне твердження  є наслідком поєднання леми \ref{L-Dzhrbsh-ident} і формули (\ref{Coushy-form}) і вперше було отримане М.М.~Джрбашяном у праці \cite{Djrbashan_1974}.

\begin{theorem}\label{T-represent-f}
Якщо функція $f \in H_2(\mathbb{C}_+),$ то для довільного $n \in \mathbb{N}$ справджується тотожність
\begin{equation}\label{represent-f}
    f(z)=\sum_{k=1}^n c_k(f)\Phi_k(z)+ \frac{B_n(z)}{2\pi {\mathrm i}}\int_{-\infty}^\infty \frac{f(t)\overline{B_n(t)}}{t-z}~\mathrm{d}t, \quad \forall z\in\mathbb C_+,
\end{equation}
де
$$
    c_k(f):=\frac{1}{\pi}\int_{-\infty}^\infty
    f(t) \overline{\Phi_k(t)}~\mathrm{d}t, \quad k=0,1,\ldots~.
$$
\end{theorem}

Нехай $A$ та $B$ -- довільні комплексні числа, $\lambda>0$ і $n$ -- будь-яке натуральне число. Покладемо
\begin{equation}\label{pi-tau}
    \nu_n(z):=\nu_n(z,{\bf a})=\prod_{k=1}^n (z-a_k), \quad
    \tau_n(z):=\tau_n(z,{\bf a})=\prod_{k=1}^n (z-\overline{a}_k).
\end{equation}
і
 \begin{equation}\label{R_n_Diff}
    R_n(z,\lambda)=R_n(z,\lambda, A,B,{\bf a})=
    \frac {A+Bz}{(z^2+\lambda^2)^2\tau_n(z)},\quad z\in {\mathbb C}^{(+)}.
\end{equation}

Позначимо через $S_n(x,R_n)$ -- частинну суму порядку $n$ ряду Фур'є функції $R_n(z,\lambda)$
 за системою  $\{\Phi_n(z)\}_{n=1}^\infty$, тобто,
  \[
  S_n(x,R_n)=\sum_{k=1}^n c_k(R_n)\Phi_k(x),
  \]
  де $c_k(R_n)=\frac 1\pi\int_{-\infty}^\infty R_n(t,\lambda)\overline{\Phi_k(t)}~{\mathrm d}t$, $k=0,1,\ldots$.

  В прийнятих позначеннях має місце таке твердження.

\begin{lemma} \label{lemma-R-S}
  Нехай $A,B$ -- довільні комплексні числа, а $\lambda>0.$ Тоді для довільних $z\in \mathbb{C}_+$ та   $n\in \mathbb{N}$ виконується тотожність
\[
    R_n(z, \lambda)-S_{n}(z; R_n)=
    \frac{1}{4\lambda^2 \tau_n ({\mathrm i}\lambda )}
    \Bigg[ \frac{{\mathrm i}A }{\lambda({\mathrm i}\lambda -z)} - \frac{A+{\mathrm i}\lambda B }{({\mathrm i}\lambda -z)^2}  + \frac{A+{\mathrm i}\lambda B }{{\mathrm i}\lambda -z}\Sigma_{n,\lambda} \Bigg]
\]
\begin{equation} \label{R-S}
    +
    \frac{1}{4\lambda^2 \nu_n (-{\mathrm i}\lambda )}\frac{\nu_n (z)}{\tau_n(z)}
    \Bigg[ \frac{{\mathrm i}A }{\lambda({\mathrm i}\lambda +z)}
    -\frac{A-{\mathrm i}\lambda B }{({\mathrm i}\lambda +z)^2}-\frac{A-{\mathrm i}\lambda B }{{\mathrm i}\lambda +z}
    \overline{\Sigma}_{n,\lambda} \Bigg],
\end{equation}
де 
\begin{equation} \label{sigma}
    \Sigma_{n,\lambda}:=\Sigma_{n,\lambda}({\bf a})=\sum_{k=1}^n \frac{a_k+{\mathrm i}\lambda}{\alpha_k^2+(\beta_k+\lambda)^2},
\end{equation}
$a_k=\alpha_k+i\beta_k$, $\alpha_k\in {\mathbb R}$, $\beta_k>0$, $k=1,2,\ldots$.

\end{lemma}

\begin{proof}

Нехай $z\in \mathbb{C}_+$. Розглянемо наступний інтеграл типу Коші
\begin{equation} \label{Koshi_I}
 {\cal J}_n(z,\lambda):=\frac 1{2\pi {\rm i}}\int_{-\infty}^\infty \frac{R_n(t,\lambda)}{t-z}{\mathrm d}t=\frac 1{2\pi {\rm i}}\int_{-\infty}^\infty \frac{A+Bt}{(t^2+\lambda^2)^2\tau_n(t)}
 \frac {{\mathrm d}t}{t-z}.
\end{equation}
Підінтегральна функція, як функція комплексної змінної $t$, має в області $\mathbb{C}_+$  простий полюс в точці $t_1=z$ та полюс кратності 2 в точці $t_2={\mathrm i}\lambda$, а при $|t|\to +\infty$ -- порядок ${\mathcal O}(|t|^{-n-5})$. Тому на основі теореми Коші про лишки
\begin{equation} \label{I_1}
 {\cal J}_n(z,\lambda)=
 \sum_{k=1}^2\mathop{\rm res}\limits_{t=t_k}
 \frac{A+Bt}{(t^2+\lambda^2)^2(t-z)\tau_n(t)}.
\end{equation}
Маємо
\begin{equation} \label{res_1}
  \mathop{\rm res}\limits_{t=t_1}
 \frac{A+Bt}{(t^2+\lambda^2)^2(t-z)\tau_n(t)}
  = \frac{A+Bz}{(z^2+\lambda^2)^2\tau_n(z)}= R_n(z, \lambda),
\end{equation}
а
\begin{equation} \label{res_2}
    \mathop{\rm res}\limits_{t=t_2}
 \frac{A+Bt}{(t^2+\lambda^2)^2(t-z)\tau_n(t)}
 =\lim\limits_{t\to {\mathrm i}\lambda}\bigg[\frac{A+Bt}{(t+{\mathrm i}\lambda)^2(t-z)\tau_n(t)} \bigg]'_t,
\end{equation}
де
\begin{equation} \label{derivative}
\bigg[\frac{A+Bt}{(t+{\mathrm i}\lambda)^2(t-z)\tau_n(t)}  \bigg]'_t=
 \frac{B}{(t+{\mathrm i}\lambda)^2(t-z)\tau_n(t)} -\frac{2(A+Bt)}{(t+{\mathrm i}\lambda)^3(t-z)\tau_n(t)}
 $$
 $$
 -\frac{A+Bt}{(t+{\mathrm i}\lambda)^2(t-z)^2\tau_n(t)}  -\frac{A+Bt}{(t+{\mathrm i}\lambda)^2(t-z) \tau_n(t)} \sum_{k=1}^n\frac 1{t-\overline{a}_k}.
\end{equation}
Звідси
\[
  \mathop{\rm res}\limits_{t=t_2}
 \frac{A+Bt}{( t^2+\lambda^2)^2(t-z)\tau_n(t)}
 = -\frac{2A}{(2{\mathrm i}\lambda)^3({\mathrm i}\lambda-z)\tau_n({\mathrm i}\lambda)}  -\frac{A+B{\mathrm i}\lambda}{(2{\mathrm i}\lambda)^2({\mathrm i}\lambda-z)^2\tau_n({\mathrm i}\lambda)}
\]
\begin{equation} \label{res_2_1}
 -\frac{A+B{\mathrm i}\lambda}{(2{\mathrm i}\lambda)^2({\mathrm i}\lambda-z) \tau_n({\mathrm i}\lambda)} \sum_{k=1}^n\frac 1{{\mathrm i}\lambda-\overline{a}_k}.
\end{equation}
Якщо $a_k=\alpha_k+\mathrm{i}\beta_k,$ то
 \[
 \frac 1{{\mathrm i}\lambda-\overline{a}_k}
 =-\frac 1{\alpha_k-{\mathrm i}(\lambda+\beta_k)}
  =-\frac {\alpha_k+{\mathrm i}\beta_k+{\mathrm i}\lambda}{\alpha_k^2+(\lambda+\beta_k)^2}
 =-\frac {a_k+{\mathrm i}\lambda}{\alpha_k^2+(\lambda+\beta_k)^2}.
 \]
 Тому, враховуючи позначення (\ref{sigma}), отримуємо
\[
  \mathop{\rm res}\limits_{t=t_2}
 \frac{A+Bt}{( t^2+\lambda^2)^2(t-z)\tau_n(t)}= - \frac{1}{4\lambda^2 \tau_n ({\mathrm i}\lambda )}
    \Bigg[ \frac{{\mathrm i}A }{\lambda({\mathrm i}\lambda -z)}
\]
\begin{equation} \label{res_2_2}
   - \frac{A+{\mathrm i}\lambda B }{({\mathrm i}\lambda -z)^2}  + \frac{A+{\mathrm i}\lambda B }{{\mathrm i}\lambda -z}\Sigma_{n,\lambda} \Bigg].
\end{equation}

Підставивши співвідношення  (\ref{res_1}) та (\ref{res_2_2}) в (\ref{I_1}), отримуємо
 \begin{equation} \label{I_1_1}
 {\cal J}_n(z,\lambda)=  R_n(z, \lambda) - \frac{1}{4\lambda^2 \tau_n ({\mathrm i}\lambda )}
    \Bigg[ \frac{{\mathrm i}A }{\lambda({\mathrm i}\lambda -z)}
 $$
 $$
  - \frac{A+{\mathrm i}\lambda B }{({\mathrm i}\lambda -z)^2}  + \frac{A+{\mathrm i}\lambda B }{{\mathrm i}\lambda -z}\Sigma_{n,\lambda} \Bigg].
\end{equation}
З іншого боку, якщо в тотожності (\ref{Dzhrbashyan-identity}) покласти $\zeta=t\in {\mathbb R}$, то використовуючи зображення
(\ref{Koshi_I}), отримаємо
 \[
  {\cal J}_n(z,\lambda)=\sum_{k=1}^{n}\frac {\Phi_k(z)}{\pi}\int_{-\infty}^\infty R_n(t,\lambda)\overline{\Phi_k(t)}{\mathrm d}t+
    \frac{B_n(z)}{2\pi {\mathrm i}}\int_{-\infty}^\infty\frac{\overline{B_n(t)}R_n(t,\lambda)}{t-z}{\mathrm d}t
 \]
 \begin{equation} \label{I_1_12}
  =S_n(z,R_n)+B_n(z)U_n(z,\lambda),
\end{equation}
де
 \begin{equation} \label{I_1_12}
   U_n(z,\lambda)=\frac{1}{2\pi {\mathrm i}}\int_{-\infty}^\infty\frac{\overline{B_n(t)}R_n(t,\lambda)}{t-z}{\mathrm d}t
   $$
   $$
   =
   \frac {\prod\limits_{k=1}^n\overline{\chi}_k}{2\pi {\rm i}}\int_{-\infty}^\infty
   \overline{\Bigg(\frac {\nu_n(t)}{\tau_n(t)}\Bigg)}\frac{A+Bt}{(t^2+\lambda^2)^2\tau_n(t)}
 \frac {{\mathrm d}t}{t-z}.
\end{equation}
Оскільки $t\in {\mathbb R}$, то
 \[
 \overline{\Bigg(\frac {\nu_n(t)}{\tau_n(t)}\Bigg)}=\frac {\prod\limits_{k=1}^n\overline{(t-a_k)}}
 {\prod\limits_{k=1}^n\overline{(t-\overline{a}_k)}}=\frac {\prod\limits_{k=1}^n(t-\overline{a}_k)}
 {\prod\limits_{k=1}^n (t- a_k)}=\frac {\tau_n(t)}{\nu_n(t)}
 \]
і тому
 \begin{equation} \label{I_1_123}
   U_n(z,\lambda) =
   \frac {\prod\limits_{k=1}^n\overline{\chi}_k}{2\pi {\rm i}}\int_{-\infty}^\infty
   \frac{A+Bt}{(t^2+\lambda^2)^2\ \nu_n(t)}
 \frac {{\mathrm d}t}{t-z}.
\end{equation}
Підінтегральна функція, як функція комплексної змінної $t$ має в області $\mathbb{C}_-$    полюс   кратності 2 в точці $t=-{\mathrm i}\lambda$, а при $|t|\to +\infty$ -- порядок ${\mathcal O}(|t|^{-n-5})$. Тому на основі теореми Коші про лишки
\begin{equation} \label{I_1}
 U_n(z,\lambda) =-\prod\limits_{k=1}^n\overline{\chi}_k\mathop{\rm res}\limits_{t=-{\mathrm i}\lambda}
 \frac{A+Bt}{(t^2+\lambda^2)^2(t-z)\nu_n(t)}.
\end{equation}
Аналогічно до (\ref{res_2_1})
\begin{equation} \label{res_3}
  \mathop{\rm res}\limits_{t=-{\mathrm i}\lambda} \frac{A+Bt}{(t^2+\lambda^2)^2(t-z)\nu_n(t)}
  =
 \frac{A}{4{\mathrm i}\lambda^3({\mathrm i}\lambda+z) \nu_n(-{\mathrm i}\lambda)}
 $$ $$
 +
 \frac{A-{\mathrm i}\lambda B}{4\lambda^2({\mathrm i}\lambda+z)^2 \nu_n(-{\mathrm i}\lambda)}  +\frac{A-{\mathrm i}\lambda B}{4\lambda^2({\mathrm i}\lambda+z)  \nu_n(-{\mathrm i}\lambda)}
 \sum_{k=1}^n\frac 1{{\mathrm i}\lambda+ {a}_k}.
\end{equation}
Якщо $a_k=\alpha_k+i\beta_k,$ то
 \[
 \frac 1{{\mathrm i}\lambda+ {a}_k} 
 =\frac 1{\alpha_k+{\mathrm i}(\lambda+\beta_k)}
  =\frac {\alpha_k-{\mathrm i}\beta_k-{\mathrm i}\lambda}{\alpha_k^2+(\lambda+\beta_k)^2}
 = \frac {\overline{a}_k-{\mathrm i}\lambda}{\alpha_k^2+(\lambda+\beta_k)^2}.
 \]
 Отже, з врахуванням (\ref{sigma}) маємо
 \begin{equation} \label{U_n_eq1}
  U_n(z,\lambda) = \frac{\prod\limits_{k=1}^n\overline{\chi}_k}{4\lambda^2 \nu_n (-{\mathrm i}\lambda )}
    \Bigg[ \frac{{\mathrm i}A }{\lambda({\mathrm i}\lambda +z)}
    -\frac{A-{\mathrm i}\lambda B }{({\mathrm i}\lambda +z)^2}-\frac{A-{\mathrm i}\lambda B }{{\mathrm i}\lambda +z}
    \overline{\Sigma}_{n,\lambda} \Bigg].
\end{equation}

З рівностей (\ref{I_1_12})  та (\ref{U_n_eq1}) отримуємо
\[
  {\cal J}_n(z,\lambda)=S_n(z,R_n)+\frac{\nu_n(z)}{\tau_n(z)}\cdot
  \frac{1}{4\lambda^2 \nu_n (-{\mathrm i}\lambda )}
    \Bigg[ \frac{{\mathrm i}A }{\lambda({\mathrm i}\lambda +z)}
    \]
\begin{equation} \label{U_n_eq}
  -\frac{A-{\mathrm i}\lambda B }{({\mathrm i}\lambda +z)^2}-\frac{A-{\mathrm i}\lambda B }{{\mathrm i}\lambda +z}
    \overline{\Sigma}_{n,\lambda} \Bigg] .
\end{equation}
Співставляючи рівності  (\ref{I_1_1}) та (\ref{U_n_eq}), отримуємо (\ref{R-S}).

\end{proof}

Для фіксованого значення $\lambda (\lambda>0)$ визначимо  поліном $T_{n-1}(z; \lambda)$
порядку $n-1$ від  $z$ за допомогою рівності
\begin{equation}\label{T}
    T_{n-1}(z; \lambda):=\tau_n (z) S_n(z; R_n), \quad n=1,2,\ldots~.
\end{equation}

Разом з $T_{n-1}(z; \lambda)$ розглянемо поліноми
\begin{equation}\label{L}
L_{n-1}(z;w):=\begin{cases}
                    \frac{\nu_n(w)-\nu_n(z)}{(z-w) \nu_n(w)}, \quad
                    w \in \mathbb{C}_-,  \\
                     \frac{\tau_n(w)-\tau_n(z)}{(w-z) \tau_n(w)}, \quad
                     w \in \mathbb{C}_+;
\end{cases}
\end{equation}

\begin{equation}\label{M}
M_{n-2}(z;w):=\begin{cases} \frac{\nu_n(w)-\nu_n(z)}{(w-z)^2 \nu_n(w)}, \quad
                    w \in \mathbb{C}_-,  \\
                     \frac{\tau_n(w)-\tau_n(z)}{(w-z)^2 \tau_n(w)}, \quad
                     w \in \mathbb{C}_+;
\end{cases}
\end{equation}

Справджується наступне твердження.

\begin{lemma} \label{lemma-r-T}
  Нехай $A,B$ --  комплексні константи, а $\lambda>0.$ Тоді якщо $z\in \mathbb{C}_+,$ то для будь-якого   $n\in \mathbb{N}$ виконується тотожність
\[
    \frac{A+Bz}{(z^2+\lambda^2)^2}-T_{n-1}(z; \lambda)=
    \frac{1}{4\lambda^2}\frac{\tau_n (z)}{\tau_n (\mathrm{i}\lambda)}
    \left[ \frac{A \mathrm{i}}{\lambda(\mathrm{i}\lambda-z)} - \frac{A+\mathrm{i}\lambda B }{(\mathrm{i}\lambda-z)^2}+\right.
     \left. \frac{A+\mathrm{i}\lambda B }{ \mathrm{i}\lambda-z }
     {\Sigma}_{n,\lambda} \right]
\]
\begin{equation} \label{r-T}
     +
    \frac{1}{4\lambda^2}\frac{\nu_n (z)}{\nu_n (-\mathrm{i}\lambda)}
    \left[ \frac{A \mathrm{i}}{\lambda(\mathrm{i}\lambda+z)}-\right.
    \left.-\frac{A-\mathrm{i}\lambda B }{(\mathrm{i}\lambda+z)^2}-
    \frac{A-\mathrm{i}\lambda B }{ \mathrm{i}\lambda+z }
     \overline{\Sigma}_{n,\lambda}\right],
\end{equation}
де ${\Sigma}_{n,\lambda}$ визначається співвідношенням (\ref{sigma}). При цьому поліном
$T_{n-1}(z; \lambda)$ має зображення
\[
    T_{n-1}(z; \lambda)=\frac{1}{4\lambda^2}
    \Bigg[\Big(\frac {A{\mathrm i}}{\lambda}+ (A+\mathrm{i}\lambda B){\Sigma}_{n,\lambda} \Big)L_{n-1}(z,\mathrm{i}\lambda)
\]
\[
    +
    \Big(\frac {A{\mathrm i}}{\lambda}-(A-\mathrm{i}\lambda B)\overline{\Sigma}_{n,\lambda} \Big)L_{n-1}(z,-\mathrm{i}\lambda)
\]
\begin{equation} \label{T-L-M-N}
   -(A+\mathrm{i}\lambda B)M_{n-2}(z,\mathrm{i}\lambda)-
    (A-\mathrm{i}\lambda B) M_{n-2}(z,-\mathrm{i}\lambda)\Bigg].
\end{equation}

\end{lemma}

\begin{proof}
 Дійсно, виразивши значення $S_n(z; R_n)$ з рівності (\ref{T}) і підставивши його в тотожність (\ref{R-S}) леми \ref{lemma-R-S}, отримаємо співвідношення (\ref{r-T}). Зображення (\ref{T-L-M-N}) полінома $T_{n-1}(z, \lambda)$ випливає з розкладу
\[
    \frac{A+Bz}{(z^2+\lambda^2)^2}=\frac{A\mathrm{i}}{4\lambda^3 ({\mathrm i}\lambda-z)}+
    \frac{A\mathrm{i}}{4\lambda^3 (\mathrm{i}\lambda+z)} -
    \frac{A+\mathrm{i}\lambda B}{4\lambda^2 (\mathrm{i}\lambda-z)^2}-
    \frac{A-\mathrm{i}\lambda B}{4\lambda^2 (\mathrm{i}\lambda+z)^2}.
\]

\end{proof}

Позначимо через $P_{n-1}$ множину довільних поліномів з комплексними коефіцієнтами порядку, не вищого за $n-1.$ Будемо вважати у подальшому, що
\begin{equation}\label{rho}
  \rho_n(z)=\rho_0 \nu_n(z)=\rho_0 \prod_{k=1}^{n}(z-a_k),
\end{equation}
де $\rho_0\not=0$ -- довільна стала.

Нехай тепер $A,B, \lambda$ -- фіксовані дійсні числа, $\lambda>0$. Введемо до розгляду функціонал
\begin{equation}\label{An}
    {\mathcal F}_n (p, \lambda):= \int_{-\infty}^{\infty} \Bigg| \frac{A+Bt}{(t^2+\lambda^2)^2}-
    p(t) \Bigg|^2 \frac{\mathrm{d}t}{|\rho_n(t)|^2}, \quad p \in P_{n-1}.
\end{equation}

Виконується наступне твердження.

\begin{theorem}\label{Teo-inf-An}
  На множині $P_{n-1}$ мінімум функціонала ${\mathcal F}_n(p,\lambda)$ реалізує поліном $T_{n-1}(z,\lambda)$ з леми \ref{lemma-r-T} і при цьому справджується рівність
\begin{equation}\label{Functional}
    \inf\{{\mathcal F}_n (p, \lambda): ~ p\in P_{n-1}\}
     =\frac{2 \pi}{\lambda \mu_n}
    \Bigg[ \Big[\frac{A}{\lambda}+ \lambda B{\mathcal A}_{n,\lambda}\Big]^2+ \frac{3A^2+\lambda^2B^2}{2\lambda^2}
   $$
   $$
   +  \frac{3A^2+\lambda^2 B^2}{ \lambda}  {\mathcal B}_{n,\lambda}
   + A^2 {\mathcal A}_{n,\lambda}^2
   + (A^2+\lambda^2B^2){\mathcal B}_{n,\lambda}^2\Bigg],
\end{equation}
в якій $    \mu_n=\prod\limits_{k=1}^{n} \Big[\alpha_k^2+(\beta_k+\lambda)^2 \Big],$ $a_k=\alpha_k+i\beta_k$, $\beta_k>0$,
 \begin{equation}\label{An_Bn}  {\mathcal A}_{n,\lambda}:=
  \sum_{k=1}^n  \frac{\alpha_k}{\alpha_k^2+(\beta_k+\lambda)^2},\quad {\mathcal B}_{n,\lambda}:=
  \sum_{k=1}^n  \frac{\beta_k+\lambda}{\alpha_k^2+(\beta_k+\lambda)^2}.\end{equation}
\end{theorem}


\begin{proof}

Введемо позначення:
\begin{equation}\label{W1}
  W_1(t):=W_1(t,\lambda,A,n)=\frac{A {\mathrm i}}{\lambda\tau_n ({\mathrm i}\lambda)}\cdot \frac 1{{\mathrm i}\lambda-t};
\end{equation}
\begin{equation}\label{W2}
  W_2(t):=W_2(t,\lambda,A,B,n)=- \frac{A+{\mathrm i}\lambda B }{\tau_n ({\mathrm i}\lambda)}
  \cdot \frac 1{({\mathrm i}\lambda-t)^2} ;
\end{equation}
\begin{equation}\label{W3}
  W_3(t):=W_3(t,\lambda,A,B,n)=\frac{(A+{\mathrm i}\lambda B) \cdot{\Sigma}_{n,\lambda} }
  {\tau_n ({\mathrm i}\lambda)}\cdot \frac 1{{\mathrm i}\lambda-t};
\end{equation}
\begin{equation}\label{Q1}
  Q_j(t):=\frac{\nu_n (t)}{\tau_n(t)}\cdot   \overline{W}_j(t),\quad j=1,2,3.
\end{equation}

В прийнятих позначеннях на підставі співвідношення (\ref{R-S}) маємо
\begin{equation}\label{|R-S|2}
    \int_{-\infty}^{\infty} |R_n(t,\lambda)-S(t,R_n)|^2~\mathrm{d}t
    =
    \frac{1}{16 \lambda^4}\int_{-\infty}^{\infty}\Bigg| \sum_{j=1}^{3} \Big(W_j(t)+Q_j(t)\Big) \Bigg|^2~\mathrm{d}t
      =  $$ $$     =\frac{1}{16 \lambda^4}\int_{-\infty}^{\infty} \Bigg[ \sum_{j=1}^{3} \Big( W_j(t)+      Q_j(t)\Big) \Bigg] \Bigg[ \sum_{j=1}^{3} \Big(\overline{W_j(t)}+
       \overline{Q_j(t)}\Big)\Bigg]~\mathrm{d}t.
\end{equation}

Обчислимо інтеграл у правій частині рівності (\ref{|R-S|2}).

При довільних $t\in \mathbb{R}$ має місце рівність $\Big| \frac{\nu_n(t)}{\tau_n(t)}\Big|=1$. Крім того,
$$
    |\tau_n({\mathrm{i}\lambda})|^2=|\nu_n({-\mathrm{i}\lambda})|^2=\prod_{k=1}^{n}
    [\alpha_k^2+(\beta_k+\lambda)^2]:=\mu_n.
$$
Тому з огляду на позначення (\ref{W1})--(\ref{Q1}) при всіх $j,k=1,2,3$ маємо
 \begin{equation}\label{Zero}
    \int_{-\infty}^{\infty} Q_j(t) \overline{W_k (t)}~\mathrm{d}t=
    \int_{-\infty}^{\infty} W_j(t) \overline{Q_k (t)}~\mathrm{d}t=0,
 \end{equation}
 \begin{equation}\label{W=Q}
    \int_{-\infty}^{\infty} W_j\overline{W_k (t)}~\mathrm{d}t=
    \int_{-\infty}^{\infty} Q_k\overline{Q_j (t)}~\mathrm{d}t.
\end{equation}
Отже, для обчислення інтегралу в правій частині рівності (\ref{|R-S|2}) достатньо знайти лише інтеграли вигляду $\int_{-\infty}^{\infty} W_j\overline{W_k (t)}~\mathrm{d}t$.

Оскільки
$$
   \int_{-\infty}^{\infty} \frac{{\mathrm d}t}{t^2+\lambda^2}=\frac{\pi}{\lambda}, \quad\quad
    \int_{-\infty}^{\infty} \frac{{\mathrm d}t}{(t^2+\lambda^2)^2}=\frac{\pi}{2\lambda^3},
$$
то
\begin{equation}\label{W1W1}
    \int_{-\infty}^{\infty} W_1(t) \overline{W_1(t)}~\mathrm{d}t 
    =\frac{A^2 \pi}{\lambda^3 \mu_n},
\end{equation}
\begin{equation}\label{W2W2}
    \int_{-\infty}^{\infty} W_2(t) \overline{W_2(t)}~\mathrm{d}t 
    = \pi\frac{A^2+\lambda^2B^2}{2\lambda^3 \mu_n},
\end{equation}
а
\begin{equation}\label{W3W3}
    \int_{-\infty}^{\infty} W_3(t) \overline{W_3(t)}~\mathrm{d}t 
    =\pi\frac{A^2+\lambda^2B^2}{\lambda \mu_n}|\Sigma_{n,\lambda}|^2.
\end{equation}

Розглянемо тепер інтеграл
\[
    \int\limits_{-\infty}^{\infty} W_2(t) \overline{W_1(t)}~\mathrm{d}t
    =-\frac{(A+\mathrm{i}\lambda B)A \mathrm{i}}{ \lambda |\tau_n(\mathrm{i} \lambda)|^2}
    \int\limits_{-\infty}^{\infty} \frac{{\mathrm d}t}{(\mathrm{i}\lambda-t)^2 (\mathrm{i}\lambda+t)}.
\]

Підінтегральна функція $
\frac{1}{(\mathrm{i}\lambda-t)^2 (\mathrm{i}\lambda+t)},$ як функція комплексної змінної, має у верхній півплощині нуль $t=\mathrm{i}\lambda$ кратності 2, а
при $t\to\infty$ порядок $\mathcal{O}(|t|^{-3}).$ Застосовуючи теорему про лишки, отримуємо
\begin{equation}\label{int-w2w1}
  \int\limits_{-\infty}^{\infty} \frac{{\mathrm d}t}{(\mathrm{i}\lambda-t)^2 (\mathrm{i}\lambda+t)}
    =\frac{\pi \mathrm{i}}{2 \lambda^2},
\end{equation}
і тому
\begin{equation}\label{W2W1=}
    \int\limits_{-\infty}^{\infty} W_2(t) \overline{W_1(t)}~\mathrm{d}t =
    -\frac{(A+\mathrm{i}\lambda B)A \mathrm{i}}{ \lambda |\tau_n(\mathrm{i} \lambda)|^2} \cdot\frac{\pi \mathrm{i}}{2 \lambda^2}=
    \frac{\pi A(A+\mathrm{i}\lambda B)}{2 \lambda^3 \mu_n}.
\end{equation}

Перейдемо до обчислення інтегралу
\[
    \int\limits_{-\infty}^{\infty} W_1(t) \overline{W_2(t)}~\mathrm{d}t
    =-\frac{(A-\mathrm{i}\lambda B)A \mathrm{i}}{ \lambda |\tau_n(\mathrm{i} \lambda)|^2}
    \int\limits_{-\infty}^{\infty} \frac{{\mathrm d}t}{(\mathrm{i}\lambda-t) (\mathrm{i}\lambda+t)^2}.
\]
Підінтегральна функція $
\frac{1}{(\mathrm{i}\lambda-t) (\mathrm{i}\lambda+t)^2},$ як функція комплексної змінної, має у верхній півплощині простий нуль $t=\mathrm{i}\lambda$, а при $t\to\infty$ порядок $\mathcal{O}(|t|^{-3}).$ Тоді аналогічно отримуємо
\begin{equation}\label{int-w1w2}
  \int\limits_{-\infty}^{\infty} \frac{{\mathrm d}t}{(\mathrm{i}\lambda-t) (\mathrm{i}\lambda+t)^2}
   =\frac{\pi \mathrm{i}}{2 \lambda^2},
\end{equation}
звідки 
\begin{equation}\label{W1W2=}
    \int\limits_{-\infty}^{\infty} W_1(t) \overline{W_2(t)}~\mathrm{d}t =
    \frac{\pi A(A-\mathrm{i}\lambda B)}{2 \lambda^3 \mu_n}.
\end{equation}

Далі знаходимо
\begin{equation}\label{W3W1}
    \int\limits_{-\infty}^{\infty} W_3(t) \overline{W_1(t)}~\mathrm{d}t
    =\frac{(A+\mathrm{i}\lambda B)A \mathrm{i}}{ \lambda |\tau_n(\mathrm{i} \lambda)|^2}\cdot
    \Sigma_{n,\lambda}\cdot
    \int\limits_{-\infty}^{\infty} \frac{{\mathrm d}t}{(\mathrm{i}\lambda-t) (\mathrm{i}\lambda+t)}=
$$
$$
    =\frac{(A+\mathrm{i}\lambda B)A \mathrm{i}}{ \lambda |\tau_n(\mathrm{i} \lambda)|^2}\cdot
    \Sigma_{n,\lambda}\cdot
    \int\limits_{-\infty}^{\infty} \frac{-{\mathrm d}t}{t^2+\lambda^2}
    =-\frac{\pi(A+\mathrm{i}\lambda B)A \mathrm{i}}{ \lambda^2 \mu_n}\cdot
    \Sigma_{n,\lambda},
\end{equation}
і, аналогічно,
\begin{equation}\label{W1W3}
    \int\limits_{-\infty}^{\infty} W_1(t) \overline{W_3(t)}~\mathrm{d}t
    =\frac{(A-\mathrm{i}\lambda B)A \mathrm{i}}{ \lambda |\tau_n(\mathrm{i} \lambda)|^2}\cdot
    \overline{\Sigma}_{n,\lambda}\cdot
    \int\limits_{-\infty}^{\infty} \frac{{\mathrm d}t}{(\mathrm{i}\lambda-t) (-\mathrm{i}\lambda-t)}=
$$
$$
    =\frac{(A-\mathrm{i}\lambda B)A \mathrm{i}}{ \lambda |\tau_n(\mathrm{i} \lambda)|^2}\cdot
    \overline{\Sigma}_{n,\lambda}\cdot
    \int\limits_{-\infty}^{\infty} \frac{{\mathrm d}t}{t^2+\lambda^2}
    =\frac{\pi(A-\mathrm{i}\lambda B)A \mathrm{i}}{ \lambda^2 \mu_n}
    \cdot\overline{\Sigma}_{n,\lambda},
\end{equation}

Нарешті, з урахуванням рівностей (\ref{int-w2w1}) i (\ref{int-w1w2}), знаходимо
\begin{equation}\label{W3W2}
    \int\limits_{-\infty}^{\infty} W_3(t) \overline{W_2(t)}~\mathrm{d}t
    =-\frac{(A^2+\lambda^2 B^2)}{  |\tau_n(\mathrm{i} \lambda)|^2}
    \cdot\Sigma_{n,\lambda}\cdot
    \int\limits_{-\infty}^{\infty} \frac{{\mathrm d}t}{(\mathrm{i}\lambda-t) (\mathrm{i}\lambda+t)^2}=
$$
$$
    =- \frac{\pi(A^2+\lambda^2 B^2)}{ 2\lambda^2 \mu_n}\mathrm{i}\cdot
    \Sigma_{n,\lambda}
\end{equation}
і
\begin{equation}\label{W2W3}
    \int\limits_{-\infty}^{\infty} W_2(t) \overline{W_3(t)}~\mathrm{d}t
    =-\frac{(A^2+\lambda^2 B^2)}{  |\tau_n(\mathrm{i} \lambda)|^2}
    \cdot\overline{\Sigma}_{n,\lambda}\cdot
    \int\limits_{-\infty}^{\infty} \frac{-{\mathrm d}t}{(\mathrm{i}\lambda-t)^2 (\mathrm{i}\lambda+t)}=
$$
$$
    = \frac{\pi(A^2+\lambda^2 B^2)}{ 2\lambda^2 \mu_n}\mathrm{i}
    \cdot\overline{\Sigma}_{n,\lambda}.
\end{equation}

Враховуючи позначення (\ref{sigma}) та (\ref{An_Bn}), бачимо, що
 $$
 |{\Sigma}_{n,\lambda}|^2={\mathcal A}_{n,\lambda}^2+{\mathcal B}_{n,\lambda}^2,
 $$
$$
    \frac{\pi(A-\mathrm{i}\lambda B)A \mathrm{i}}{ \lambda^2 \mu_n}\cdot
    \overline{\Sigma}_{n,\lambda}-
    \frac{\pi(A+\mathrm{i}\lambda B)A \mathrm{i}}{ \lambda^2 \mu_n}\cdot
    \Sigma_{n,\lambda}
    =\frac{2\pi AB}{\lambda \mu_n} \cdot{\mathcal A}_{n,\lambda}+
    \frac{2\pi A^2}{\lambda^2 \mu_n}\cdot {\mathcal B}_{n,\lambda}
$$
i
$$
    \frac{\pi(A^2+\lambda^2 B^2)}{ 2\lambda^2 \mu_n}\cdot\mathrm{i}
    \cdot\overline{\Sigma}_{n,\lambda}-
    \frac{\pi(A^2+\lambda^2 B^2)}{ 2\lambda^2 \mu_n}\cdot\mathrm{i}
    \cdot\Sigma_{n,\lambda}
    =\frac{\pi(A^2+\lambda^2 B^2)}{ \lambda^2 \mu_n} \cdot{\mathcal B}_{n,\lambda}.
$$
Тому об'єднуючи співвідношення (\ref{|R-S|2}--(\ref{W3W3}), (\ref{W2W1=}) і (\ref{W1W2=})--(\ref{W2W3}),  одержуємо рівність
\begin{equation}\label{|R-S|2=}
    \int_{-\infty}^{\infty} |R_n(t,\lambda)-S(t,R_n)|^2~\mathrm{d}t=\frac{2 \pi}{\lambda \mu_n}
    \Bigg[\frac{2A^2 }{\lambda^2}+ \frac{A^2+\lambda^2B^2}{2\lambda^2}
   $$
   $$
   +2AB {\mathcal A}_{n,\lambda}+ \frac{2A^2}{\lambda}{\mathcal B}_{n,\lambda}+\frac{A^2+\lambda^2 B^2}{ \lambda}  {\mathcal B}_{n,\lambda}
   + (A^2+\lambda^2B^2)\Big[{\mathcal A}_{n,\lambda}^2+{\mathcal B}_{n,\lambda}^2\Big]\Bigg]
   $$
   $$
   =\frac{2 \pi}{\lambda \mu_n}
    \Bigg[ \Big[\frac{A}{\lambda}+ \lambda B{\mathcal A}_{n,\lambda}\Big]^2+ \frac{3A^2+\lambda^2B^2}{2\lambda^2}
   $$
   $$
   +  \frac{3A^2+\lambda^2 B^2}{ \lambda}  {\mathcal B}_{n,\lambda}
   + A^2 {\mathcal A}_{n,\lambda}^2
   + (A^2+\lambda^2B^2){\mathcal B}_{n,\lambda}^2\Bigg].
\end{equation}

Зауважимо, що будь-який поліном $p\in P_{n-1}$ може бути зображений у вигляді
\[
    p(t)=\tau_n(t) \sum_{k=1}^{n} b_k \Phi_k(t),
\]
з певними коефіцієнтами $\{b_k\}_{k=1}^n$, і навпаки, для будь-якого набору коефіцієнтів
$\{b_k\}_{k=1}^n$ вираз з правої частини останьої рівності є певним поліномом $p\in P_{n-1}.$ Отже, беручи до уваги визначення функції $R_n(z, \lambda)$ (\ref{R_n_Diff}) -- (\ref{pi-tau}), можна стверджувати, що функціонал ${\mathcal F}_n(p, \lambda)$ має зображення
\begin{equation}\label{A_n-represent}
    {\mathcal F}_n(p, \lambda)=|\rho_0|^2 \int_{-\infty}^{\infty}
    \Big|R_n(t, \lambda)- \sum_{k=1}^{n} b_k \Phi_k(t)\Big|^2 \mathrm{d}t
\end{equation}
з певними значеннями $\{b_k\}_{k=1}^n$, і навпаки, для будь-якого набору сталих
$\{b_k\}_{k=1}^n$ вираз з правої частини  рівності (\ref{A_n-represent}) є значенням функціоналу для певного полінома $p\in P_{n-1}.$

Але згідно з формулою  (\ref{R-S}) леми \ref{lemma-R-S} величина $S_n(t, R_n)$ є частинною сумою порядку $n$ ряду Фур'є функції $R_n(z, \lambda)$ по ортогональній на всій дійсній осі системі $\{\Phi_k(t)\}_{k=1}^\infty.$ Тому з рівностей (\ref{|R-S|2=}) і
(\ref{A_n-represent}) випливає, що для довільного поліному $p\in P_{n-1}$
\begin{equation}\label{A_n>}
    {\mathcal F}_n(p, \lambda) \ge \int_{-\infty}^{\infty} \Bigg| \frac{A+Bt}{(t^2+\lambda^2)^2}-
    T_{n-1}(t, \lambda) \Bigg|^2 \frac{\mathrm{d}t}{|\rho(t)|^2}=\frac{2 \pi}{\lambda \mu_n}
    \Bigg[ \Big[\frac{A}{\lambda}+ \lambda B{\mathcal A}_{n,\lambda}\Big]^2
   $$
   $$
  + \frac{3A^2+\lambda^2B^2}{2\lambda^2} +  \frac{3A^2+\lambda^2 B^2}{ \lambda}  {\mathcal B}_{n,\lambda}
   + A^2 {\mathcal A}_{n,\lambda}^2
   + (A^2+\lambda^2B^2){\mathcal B}_{n,\lambda}^2\Bigg].
\end{equation}

При цьому, знак рівності можливий лише у тому випадку, коли
\[
    \sum_{k=1}^{n} b_k \Phi_k(t)=S_n(t, \lambda),
\]
або, що те саме, $p(t)=T_{n-1}(t, \lambda).$

\end{proof}

\bibliographystyle{plain}
\renewcommand{\refname}{Література}

\end{document}